\documentclass[10pt,draft,reqno]{amsart}
\usepackage{amsmath,amssymb,euscript,amsthm,amsfonts}

\usepackage{amsmath,amssymb,euscript,amsthm,amsfonts}
\newcommand{\1}{1\!\!\!\;{\rm I}}
\newcommand{\mbR}{{\mathbb R}}

\newcommand{\pt}{\partial}
\newcommand{\wt}{\widetilde}

\newcommand{\ov}{\overline}

\newcommand{\vf}{\varphi}
\newcommand{\ve}{\varepsilon}

\newcommand{\cE}{\mathcal E}

\makeatletter
\@addtoreset{equation}{section}
\makeatother
\makeatletter
\@addtoreset{corl}{section}
\makeatother
\makeatletter
\@addtoreset{expl}{section}
\makeatother
\makeatletter
\@addtoreset{thm}{section}
\makeatother
\makeatletter
\@addtoreset{lem}{section}
\makeatother
\makeatletter
\@addtoreset{defn}{section}
\makeatother

\theoremstyle{plain}
\newtheorem{thm}{Theorem}
\newtheorem{lem}{Lemma}

\newtheorem{corl}{Corollary}
\theoremstyle{definition}

\theoremstyle{remark}

\begin{document}

\title[Remarks on differentiability in the initial data
for stochastic reflecting flow
 ]{Remarks on differentiability in the initial data
for stochastic reflecting flow}

\author{Andrey Pilipenko}
\thanks{Research is partially supported by State fund for fundamental researches of Ukraine and the Russian foundation for basic researches,
    Grant F40.1/023}
\thanks{Research is partially supported by the Grant of the President of Ukraine,
    Grant F47/033}
\address{Andrey Pilipenko: Institute of Mathematics of Ukrainian National Academy of Sciences, Tereschenkivska 3, 01601, Kiev-4, Ukraine}
\email{pilipenko.ay@yandex.ua}
\urladdr{http://imath.kiev.ua/ $\wt{ }$ apilip}

\subjclass[2000] {Primary 58J65; Secondary 60H10, 60H25}

\keywords{Stochastic reflecting flow. Differentiability of a
stochastic flow}

\begin{abstract}

Stochastic flows generated by reflected SDEs in a half-plane with an additive diffusion term  are considered.  A derivative in the initial data is represented a.s.  as an infinite product of matrices. We use this representation and  construct  an example of a reflecting flow with a linear drift such that it is not locally continuously differentiable.
\end{abstract}

\maketitle

\textbf{Introduction}

Differentiability in the initial data of flows generated by SDEs with smooth coefficients  is well-studied subject of stochastic analysis (see for example  \cite{IW, Ku}). Equations for derivatives can be  obtained by formal differentiation of initial equations. A problem of differentiability of flows generated by reflected SDEs (RSDEs) in a domain is more complicated. The  corresponding results on differentiability appeared comparatively recently.

The first paper on reflecting flows differentiability  is due to Deuschel and Zambotti \cite{DZ}. They considered reflecting flows in an orthant with additive diffusion term.They considered reflecting flows in an orthant with additive diffusion term. The approach of \cite{DZ} was generalized by Andres \cite{Andres1, Andres2} to SDEs (with additive noise) in a polyhedron or a domain with a smooth boundary. Another approach was developed by Pilipenko in   \cite{PilDAN, PilTBiMC, PilHaus,
Pil_Cosa}, where the Sobolev derivatives were studied. It also was noticed that if we are able to prove that the flow satisfies the Lipschitz property with respect to the initial data, then this implies not only Sobolev but also Frechet differentiability \cite{Pil_Cosa}. Usually the Lipschitz  property is satisfied  if a diffusion term is  constant and a drift is Lipschitzian. If the diffusion term is not  constant then the problem of Frechet differentiability is open even for $C^\infty$ coefficients.
The third approach  of investigation was proposed by Burdzy \cite{Burdzy}, who  used excursion theory to study a reflected Brownian  flow in a domain with a smooth boundary. It is worth to mention that the curvature of the boundary gives some new interesting terms  to a representation for the derivative,  that  contain a local time of a process at the boundary.

Consider an SDE in a half-spase $\mbR^d_+=\mbR^{d-1}\times[0, \infty)$ with a normal reflection at the boundary $\pt\mbR^d_+=\mbR^{d-1}\times\{0\}:$
\begin{align}
&
d\vf_t(x)=a(\vf_t(x))dt+dw(t)+nL(dt, x), \ t\geq0,
\label{eq6.1}\\
&
\vf_0(x)=x, \ x\in\mbR^d_+,\ \vf_t(x)\in \mbR^d_+, t\geq 0,
\label{eq6.2}
\end{align}
where $n=(0, \ldots, 0, 1)$ is a  normal vector to the hyperplane $\pt\mbR^d_+,$ $ w(t)$  is a Wiener process
 in  $\mbR^d$,
\begin{equation}
\label{eq6.3}
 \{L(t, x), t\geq0\} \
\mbox{ is continuous and non-decreasing in $t$ process},
\end{equation}
\begin{equation}
\label{eq6.4}
L(0, x)=0,
\end{equation}
\begin{equation}
\label{eq6.5}
\int^t_0\1_{\vf_s(x)\in\pt\mbR^d_+}L(ds, x)=L(t, x), \ t\geq0.
\end{equation}
The  last condition means that $L(t, x)$ does not increase in $t $ when $\vf_t(x)\in  \mbR^d_+\setminus\pt\mbR^d_+.$ So, a solution of the  RSDE behaves as a solution of an SDE without reflection inside the upper half-space.

Let us give informal explanation how to guess a form of an equation for the derivative in initial data (the proof of the corresponding fact is non-trivial). Since inside the upper half-space the equation behaves as usual SDE, the derivative should be obtained by formal differentiation of \eqref{eq6.1} with respect to $x$, i.e., $\pt\nabla\vf_t(x)/\pt t=\nabla a(\vf_t(x))\nabla\vf_t(x)$ if $\vf_t(x)\notin\pt\mbR^d_+.$ If $\vf_t(x)\in\pt\mbR^d_+$, then  the $d$-th coordinate of the process $\vf_t(x)$ attains a minimum (it equals zero). So, the derivative of $d$-th coordinate should be equal to 0. This requirements and some technical details are enough to determine uniquely  the derivative \cite{ Andres1, Andres2, DZ, PilTBiMC, Pil_Cosa} (see \S1 for strict statement).

Underline one important circumstance. It was proved in all  papers cited above that
\begin{equation}
\label{eq2.5!}
\forall \ x \ \forall \ t \ \ \  P(\exists\nabla\vf_t(x))=1.
\end{equation}
Note that the statements
$$
\forall \ x \ \  P(\forall \ t\geq0 \ \ \exists\nabla\vf_t(x))=1
$$
and
$$
\forall \ t\geq0 \ \  P(\forall \ x \ \ \exists\nabla\vf_t(x))=1,
$$
are, generally, incorrect. This fact is easy to explain in one-dimensional case. Let
$\vf_t(x)$ be a reflected Brownian motion. It can be checked  that
$$
\vf_t(x)=\begin{cases}
\vf_t(0)=w(t)-\min_{0\leq s\leq t}w(s), t>\sigma(x),\\
x+w(t), t\leq \sigma(x),
\end{cases}
$$
where $\sigma(x)$ is the first instant when $x+w(\cdot)$   hits zero:
$$
\sigma(x)=\inf\{t\geq0: x+w(t)=0\}.
$$
Then
$$
\frac{\pt\vf_t(x)}{\pt x}=
\begin{cases}
0, x<-\min_{0\leq s\leq t}w(s),\\
1, x>-\min_{0\leq s\leq t}w(s),\\
\mbox{does not exist}, x=-\min_{0\leq s\leq t}w(s).
\end{cases}
$$
However, for any fixed $t>0, x>0:$
\begin{equation}
\label{eq4.1}
P(\vf_t \ \mbox{is continuously differentiable in a neighborhood of} \ x)=1.
\end{equation}

This example makes reasonable a conjecture: \emph{equality \eqref{eq4.1} is always satisfied if coefficients  of the RSDE  are smooth.}  A  result of paper \cite{DZ} gives another argument in favor of this hypothesis. It was proved (additive noise, $C^1$ drift, normal reflection at hyperplanes) that there exists a modification $\psi_t(x)$  of the derivative $\nabla\vf_t(x)$  such that for all $t, x_0$
\begin{equation}
\label{eq4.2!}
 P\left(\psi_t(x_0)={\nabla\vf_t(x_0)}\right)=1,
\end{equation}
and
\begin{equation}
\label{eq4.2}
 \psi_t(x)\to\psi_t(x_0), \ x\to x_0 \
\mbox{a.s.}
\end{equation}

In \S\,2 we give an example of a flow $\vf_t(x), t\geq0, x\in\mbR^d_+,$ generated by RSDE in a half-plane $\mbR^d_+$ with normal reflection at the boundary $\pt\mbR^d_+$, additive diffusion and $C^\infty$ drift such that it is  not locally continuously differentiable flow a.s.It will be shown that this flow is not even locally differentiable   a.s.  Moreover
$$
\forall \ x \ \
P\left(\{\sigma(x)<t\}\cap\{\exists \ \mbox{a neighbourhood} \ U(x) \
\forall \ y\in U(x) \ \exists \ {\nabla\vf_t(y)}\}\right)=0,
$$
where $\sigma(x)=\inf\{t\geq0: \ \vf_t(x)\in\pt \mbR^d_+\}$.

Note that this statement neither contradicts \eqref{eq2.5!} nor contradicts \eqref{eq4.2!}, \eqref{eq4.2}.

In \S\,1 we give some preliminary formulas, in particular, the derivative
$\nabla\vf_t$  will be represented as an infinite product of matrices.

\textbf{\S\,1. Representation of the derivative on initial value for reflecting flow}

To give a representation of the derivative on the initial value for a reflecting stochastic flow, we need to introduce one type of integral equation (see Theorem \ref{thm2} below).  An equation for the derivative on the initial value is given in Theorem \ref{thm1}. The main result of this Section is Theorem \ref{thm4}, where we  express the derivative
as an infinite products of matrices.

Consider $d\times d$ matrices

$P=\begin{pmatrix}
1&0&\ldots&0\\
0&\ddots&&\vdots\\
&\ddots&1&0\\
0&&0&0
\end{pmatrix}$, $Q=E-P=\begin{pmatrix}
0&0&\ldots&0\\
0&\ddots&&\vdots\\
&\ddots&0&0\\
0&&0&1
\end{pmatrix}$.

\begin{thm}
\label{thm2}
Let  $\alpha: [0, \infty)\to\mbR^d\otimes\mbR^d$  be a continuous and bounded function taking values in the space of  $d\times d$ matrices, and
$\beta: [0, \infty)\to\mbR$ a continuous function, $\beta(0)\ne0.$
Denote
\begin{equation}
\label{eq9.0}
\sigma=\inf\{t\geq0: \ \beta(t)=0\},\ 
\tau(t)=\sup\{s\in[0, t]: \ \beta(s)=0\}.
\end{equation}
Then there exists a unique function $\gamma: [0, \infty)\to\mbR^d\otimes
\mbR^d$ satisfying the system
\begin{equation}
\label{eq9.1}
P\gamma(t)=P+\int^t_0P\alpha(s)\gamma(s)ds, \ t\geq0,
\end{equation}
\begin{equation}
\label{eq9.2}
Q\gamma(t)=\begin{cases}
Q+\int^t_0Q\alpha(s)\gamma(s)ds,  \ t<\sigma,\\
\int^t_{\tau(t)}Q\alpha(s)\gamma(s)ds, \ t\geq\sigma.
\end{cases}
\end{equation}
\end{thm}

Theorem on existence and uniqueness  for a solution of such  type equation
 was proved in more general setting in
\cite{ Andres2,   IW}, 
see also
\cite{Airault}, where such equations were introduced for the first time.

So, we only give a sketch of a proof in order to explain
   difficulties that may arise and a form of representations that we will obtain.  

Assume at first that  the function $\beta$ has only finite number of zeroes
$\sigma=\sigma_0<\sigma_1<\sigma_2<\ldots<\sigma_n,$  then we can solve
\eqref{eq9.1}, \eqref{eq9.2} successively on intervals
$[0, \sigma_0), [\sigma_0, \sigma_1), [\sigma_1, \sigma_2) $  and so on.
It is easy to see that in this case   $\gamma(t)$ is obtained by successive solution of the following linear equations
\begin{equation}
\label{eq10.1}
\gamma(t)=E+\int^t_0\alpha(s)\gamma(s)ds, \ t<\sigma_0,
\end{equation}
\begin{equation}
\label{eq10.2}
\gamma(t)=P\gamma(\sigma_{k}-)+\int^t_{\sigma_k}
\alpha(s)\gamma(s)ds, \ t\in[\sigma_k, \sigma_{k+1}).
\end{equation}

Denote by $\cE_{st}$ a solution of the following matrix-valued equation
\begin{equation}
\label{eq10.20}
\begin{cases}
\frac{\pt\cE_{st}}{\pt t}=\alpha(t)\cE_{st}, \ t\geq s,\\
\cE_{ss}=E.
\end{cases}
\end{equation}

It follows from \eqref{eq10.1}, \eqref{eq10.2} that
\begin{equation}
\label{eq10.30}
\gamma(t)=\cE_{0t}, \ t<\sigma_0.
\end{equation}

If $t\in[\sigma_k, \sigma_{k+1}),$  then
\begin{equation}
\label{eq10.3}
\begin{split}
\gamma(t)=
\cE_{\sigma_kt}P\gamma(\sigma_{k^-})=
\cE_{\sigma_kt}P
\cE_{\sigma_{k-1}\sigma_k}
P\gamma(\sigma_{{k-1}^-})= \ldots\\
=
\cE_{\sigma_kt}P
\cE_{\sigma_{k-1}\sigma_k}
P\ldots P
\cE_{\sigma_{1}\sigma_0}
P
\cE_{0\sigma_0}.
\end{split}
\end{equation}

Assume now that  $\beta$ has infinite number of zeroes, for example, let $\beta$  be a typical trajectory of a Wiener process. Then we should use more delicate methods to solve
\eqref{eq9.1}, \eqref{eq9.2}. It can be done as follows. Introduce a function $\pi$ that takes a matrix-valued function   $x=x(t), t\geq0,$  to
\begin{equation}
\label{eq11.0}
(\pi x)(t)=
\begin{cases}
x(t), t<\sigma,\\
Px(t)+Q(x(t)-x(\tau(t))), t\geq\sigma,
\end{cases}
\end{equation}
where $\sigma$ and $\tau(t)$ are from \eqref{eq9.0}.

Observe that  system
\eqref{eq9.1}, \eqref{eq9.2}
is equivalent to the following
\begin{equation}
\label{eq11.1}
\gamma(t)=\pi\left(E+\int^\cdot_0\alpha(s)\gamma(s)ds\right)(t), \ t\geq0.
\end{equation}
It is easy to see that for any $T>0:$
\begin{equation}
\label{eq11.2}
\sup_{t\in[0, T]}
\|\pi x_1(t)-\pi x_2(t)\|
\leq 2\sup_{t\in[0, T]}
\|x_1(t)-x_2(t)\|.
\end{equation}

Consider successive approximations
$$
\begin{aligned}
&
\gamma_0(t)=E, \ t\geq0, \\
&
\gamma_{n+1}(t)=\pi\left(E+\int^\cdot_0\alpha(s)\gamma_n(s)ds\right)(t).
\end{aligned}
$$
From \eqref{eq11.2} and standard reasoning we can conclude that there exists a unique solution of
\eqref{eq11.1}, and successive  approximations converge to this solution
$$
\gamma(t)=\lim_{n\to\infty}\gamma_n(t).
$$

The set $\{t\geq0: \beta(t)=0\}$ is closed. Therefore, if  $x$ is continuous function then $\pi x$ has   $c\grave{a}dl\grave{a}g$ trajectories. Moreover, it is not hard to see that
$P\gamma(t)=E+P\int^t_0\alpha(s)\gamma(s)ds$ is continuous in $t$, and $Q \gamma(t)$ is $c\grave{a}dl\grave{a}g$.

Consider RSDE \eqref{eq6.1}-\eqref{eq6.5} in a half-space $\mbR^d_+=\mbR^{d-1}\times[0, \infty)$ with a normal reflection at the boundary $\pt\mbR^d_+=\mbR^{d-1}\times\{0\}.$
Assume that a function  $a: \mbR^d_+\to\mbR^d$ is continuously differentiable and its derivative is bounded. Then 
 for  any  $\omega $  there exists a unique solution \eqref{eq6.1}--\eqref{eq6.5} and this solution is continuous in
$(t, x)$ and uniformly Lipschitzian in  $x$  for $t\in[0, T]:$
$$
\forall \ T>0 \ \exists \ c \ \forall \ x_1, x_2\in\mbR^d_+ \ \forall \ t\in
[0, T]: \
\|\vf_t(x_1)-\vf_t(x_2)\|\leq c\|x_1-x_2\|.
$$
Really, the Skorokhod map in a half-plane is Lipschitzian. Since the noise is additive,  the existence, uniqueness for a solution of RSDE, and Lipschitz
property for any $\omega$ can be proved by standard arguments, see for example from
\cite{Anulova}.


Denote
\begin{equation}
\label{eq7.-1}
\sigma(x)=\inf\{t\geq0: \ \vf_t(x)\in\pt\mbR^d_+\}.
\end{equation}

\begin{thm}
\label{thm1}
For all $x\in\mbR^{d-1}\times(0, \infty), t\geq0$
$$
P(\mbox{Frechet derivative $\nabla\vf_t(x)$  exists})=1.
$$
Moreover there exists a modification $\psi_t(x)$ of the derivative, i.e.
\newline $P(\psi_t(x)=\nabla\vf_t(x))=1,$ $ t\geq0, x\in\mbR^{d-1}\times(0, \infty)$,

 such that

1) for any  $x$ the process $\psi_t(x), t\geq 0$ is $c\grave{a}dl\grave{a}g$,

2)
\begin{equation}
\label{eq7.0}
P\psi_t(x)=P+\int^t_0P\nabla a(\vf_s(x))\psi_s(x)ds, \ t\geq0,
\end{equation}
\begin{equation}
\label{eq7.1}
Q\psi_t(x)=\begin{cases}
Q+
\int^t_0Q\nabla a(\vf_s(x))\psi_s(x)ds, t<\sigma(x),\\
\int^t_{\tau(t, x)}Q\nabla a(\vf_s(x))\psi_s(x)ds, t\geq\sigma(x),
\end{cases}
\end{equation}
where $P$ and $Q$ are the same as in Theorem \ref{thm2},
\begin{equation}
\label{eq8.0}
\tau(t, x)=\sup\{s\in[0, t]: \ \vf_s(x)\in\pt\mbR^d_+\}
\end{equation}
is the last instant before  $t$ when the process  $\vf_\cdot(x)$ visit the hyperplane.
\end{thm}
The proof of a differentiability and representation
\eqref{eq7.0}, \eqref{eq7.1}
see for example in
\cite{Pil_Cosa, Andres2}.

{\bf Remark.}  System  \eqref{eq7.0}, \eqref{eq7.1} is a particular case of \eqref{eq9.1}, \eqref{eq9.2}, where
$\alpha(t)=
\nabla a(\vf_t(x)),$ and $\beta(t)$  is  $\vf^d_t(x)$ (the
$d$-th coordinate of the process $\vf_t(x)$).

{\bf Remark.} Assume that
$\vf_t(x)\notin\pt\mbR^d_+, t\in[t_1, t_2].$
Then there exists a neighborhood
$U(x)$  of $x$  such that
$\vf_t(y)\notin\pt\mbR^d_+, t\in[t_1, t_2], y\in U(x).$ So, $\vf_t(y)$ satisfies the following integral equation
\begin{equation}
\label{eq8.1}
\vf_t(y)=\vf_{t_1}(y)+\int^t_{t_1}a(\vf_s(y))ds+w(t)-w(t_1), \
t\in[t_1, t_2], \ y\in U(x).
\end{equation}
Representations \eqref{eq7.0}, \eqref{eq7.1}  imply that
$$
\frac{\pt\psi_t(y)}{\pt t}=\nabla a(\vf_t(y))\psi_t(y), \ t\in[t_1, t_2].
$$
In particular,
$$
\psi_t(x)=E+\int^t_0\nabla a(\vf_s(x))\psi_s(x)ds, \ t<\sigma(x),
$$
as it should be for a derivative in the initial data of integral equation \eqref{eq8.1}.

The main  aim of this Section is to obtain a  representation
of  \eqref{eq7.0}, \eqref{eq7.1} solution, which is similar to \eqref{eq10.30}, \eqref{eq10.3}.
  We prove the corresponding result in  general settings for equations \eqref{eq9.1}, \eqref{eq9.2}.

Note that a product of matrices depends on the order of the product. Thus we need a formal definition and sufficient condition for   convergence of infinite product.


Let $K$ be a countable set with a linear order $\leq.$ Let us introduce a partial order on finite subsets of  $K$  as follows
$$
L_1\preccurlyeq L_2\overset{\mbox{def}}{\Leftrightarrow}L_1\subset L_2.
$$

Let  $X$ be a Banach space, $\{A_k, k\in K\}$ be a collection of linear continuous operators on $X$   (for example,   $X=\mbR^n,$
$A_k$ is $d\times d$ matrix).

Let
$L=\{l_1, \ldots, l_n\}\subset K,  \ l_1\geq\ldots\geq l_n.$  By  $\prod_{k\in L}A_k$ denote  a product
$A_{l_1}A_{l_2}\ldots A_{l_n}$  (operators with greater indices are on the left).

\textbf{Definition 1.} An infinite product $\prod_{k\in K}A_k$  converges and equals a linear continuous operator $U$  if
$$
\forall \ \ve>0 \ \exists \ L_0\subset K,
|L_0|<\infty \ \forall
L\succcurlyeq L_0:
\ \ \|\prod_{k\in L}A_k-U\|<\ve,
$$
where $|\cdot|$ is a number of elements in a set, $\|\cdot\|$ is a norm of a linear operator.

{\bf Remark.}  Definition 1 means a convergence of generalized sequence of matrices
$\{\prod_{k\in L}A_k, |L|<\infty, L\subset K\}$, where partial order is
$\preccurlyeq.$

{\bf Remark.} We do not require the non-degeneracy of a limit in contrast to the usual definition of infinite product of numbers.

\begin{thm}
\label{thm3}
 Assume that $A_k=E+B_k, k\in K,$ where linear operators
$B_k$  are such that
$$
\sum_{k\in K}\|B_k\|<\infty.
$$
Then the infinite product $\prod_{k\in K}A_k$  converges.
\end{thm}

{\bf Remark.}  A sum of real-valued series with non-negative terms is independent of the order of summation.

\begin{proof}[Proof of Theorem \ref{thm3}]
Note that for any collections of operators $\{C_k\}, \{D_k\}$ the following inequality holds
\begin{equation}
\label{eq14.1}
\begin{split}
&\|(E+D_0)(E+C_1)(E+D_1)(E+C_2)\cdot\ldots\cdot (E+C_n)(E+D_n)-\\
&
-
(E+C_1)\cdot\ldots\cdot(E+C_n)\|=\\
&=\|\sum_k\Bigl[\Bigl((E+D_0)(E+C_1)(E+D_1)\cdot\ldots\cdot (E+C_{k-1})(E+D_{k-1})-\\
&
-(E+D_0)(E+C_1)(E+D_1)\cdot\ldots\cdot (E+C_{k-1})\Bigr)(E+D_{k})\cdot\dots(E+D_{n})\Bigr]\|\leq\\
&\leq
(\sum^n_{k=0}\|D_k\|)\prod^n_{k=0}(1+\|D_k\|)\prod^n_{k=1}(1+\|C_k\|)\leq\\
&\leq
(\sum^n_{k=0}\|D_k\|)\exp\{\sum^n_{k=0}\|D_k\|+\sum^n_{k=1}\|C_k\|\}.
\end{split}
\end{equation}
Let

\noindent $L_0\subset L_1,\ $ $L_0=\{l_1,\ldots, l_n\},$ 

\noindent $L_1=\{m_{0,1},\dots, m_{0, k_0}, l_1,
m_{1,1}, \dots, m_{1, k_1}, l_2,
m_{2,1}, \dots, m_{2, k_2}, l_n,
m_{n,1}, \dots, m_{n, k_n}
 \},$

\noindent where elements in parenthesis are in decreasing order. It follows from
\eqref{eq14.1}  that
\begin{equation}
\label{eq15.1}
\begin{split}
\|\prod_{k\in L_0}(E+B_k)-\prod_{k\in L_1}(E+B_k)\|\leq\\
\leq
(\sum^n_{i=0}\sum^{k_i}_{j=0}\|B_{m_{i,j}}\|)
\exp
\left\{
\sum^n_{i=0}\sum^{k_i}_{j=0}\|B_{m_{i,j}}\| +\sum^n_{i=1}\|B_{l_i}\|
\right\}\leq\\
\leq
\sum_{j\notin L_0}\|B_j\|\cdot\exp\left\{\sum_{k\in K}\|B_k\|\right\}.
\end{split}
\end{equation}

Let $\ve>0$ be fixed. Choose $L_0$ such that
$$
\sum_{j\notin L_0}\|B_j\|\exp\left\{\sum_{k\in K}\|B_k\|\right\}<\ve.
$$
Thus \eqref{eq15.1} implies that for any  $L_1, L_0\subset L_1:$
$$
\|\prod_{k\in L_0}A_k-\prod_{k\in L_1}A_k\|\leq \ve.
$$
So $\{\prod_{k\in L}A_k, |L|<\infty, L\subset K\}$ is a generalized Cauchy sequence. This implies  a
convergence of the product (see \cite[Ch.1,
\S\,7]{DCS}).

Theorem \ref{thm3} is proved.
\end{proof}

Let us consider equations \eqref{eq9.1}, \eqref{eq9.2}. Represent a set $\{t\geq0: \ \beta(t)\ne0\}$ as a denumerable union of disjoint sets
$[0, \sigma_0)\cup\cup_k(\sigma_k, \tau_k),$ where possibly 
$\sigma_0=\infty$  or $\tau_k=\infty$  for some $k.$  Introduce a linear order in the set  $K=\{(\sigma_k; \tau_k)\}$ of intervals:
$$
(\sigma_i, \tau_i)<(\sigma_j, \tau_j)\Leftrightarrow \tau_i<\sigma_j.
$$

The main results of this Section is the next Theorem and   Corollary.

\begin{thm}
\label{thm4}
Assume that
\begin{equation}
\label{eq15.3}
\lambda(\{t\geq0: \ \beta(t)=0\})=0,
\end{equation}
where $\lambda$ is a Lebesgue measure. Then a solution of   system \eqref{eq9.1},
\eqref{eq9.2} is of the form
\begin{equation}
\label{eq15.2}
\gamma(t)=\begin{cases}
\cE_{0t}, \ t<\sigma_0,\\
\cE_{\tau(t)\; t}P\left(
\prod_{(\sigma_i \tau_i)\subset[0, t]}(P\cE_{\sigma_i \tau_i}P)
\right)\cE_{0 \sigma_0}, t\geq\sigma_0,
\end{cases}
\end{equation}
where $\cE_{s t}$  is defined in \eqref{eq10.20}.
\end{thm}
{\bf Remark.} Notice that  $\lambda(\{t\geq0: \ \vf^d_t(x)=0\})=0$ a.s., where
$\vf^d_t(x)$ is $d$-th coordinate of the process $\vf_t(x).$  So, conditions of Theorem \ref{thm4} are satisfied for the solution of \eqref{eq7.0}, \eqref{eq7.1}
for a.a. $\omega$.  
Combining Theorems \ref{thm1} and \ref{thm4}  we obtain the following statement.

\begin{corl}
\label{corl11}
Let $\vf_t(x)$ be a solution of \eqref{eq6.1} -- \eqref{eq6.5}. Then for all $t\geq0, x\in\mbR^{d-1}\times(0, \infty)$  with probability 1 we have
\begin{equation}
\label{eq16.0}
\nabla\vf_t(x)=\begin{cases}
\cE_{0 t}(x), \ t<\sigma(x), \\
\cE_{\tau(t, x)\; t}(x)
P\prod_{(\sigma_k(x), \tau_k(x))\subset[0, t]}
(P\cE_{\sigma_k(x) \tau_k(x)}P)\cE_{0 \sigma(x)}, \ t\geq\sigma(x),
\end{cases}
\end{equation}
where $\cE_{st}(x)$ is a solution of
\begin{equation}
\label{eq16.1}
\begin{cases}
\frac{\pt}{\pt t}\cE_{s t}(x)=\nabla a(\vf_t(x))\cE_{s t}(x), t\geq s,\\
\cE_{s s}(x)=E,
\end{cases}
\end{equation}
$\sigma(x), \tau(t, x)$  are defined in  \eqref{eq7.-1} and \eqref{eq8.0} respectively, and
$\{(\sigma_k(x), \tau_k(x))\}$ is a collection of disjoint intervals such that
\begin{equation}
\label{eq_dop32}
\cup_k(\sigma_k(x), \tau_k(x))=\{t>\sigma(x): \ \vf_t(x)\notin\pt\mbR^d_+\}.
\end{equation}
\end{corl}

\begin{proof}[Proof of Theorem \ref{thm4}]
The result of Theorem \ref{thm4} is obvious if a number  of intervals
$(\sigma_k, \tau_k)$ is finite (see representation  \eqref{eq10.3} and observe that  $P^2=P$). Therefore, further we consider only the case when the corresponding number of intervals is countable.

Select a sequence  $\{K_n\}_{n\geq1}\subset K,$
$K_n=\{(\sigma^{(n)}_i, \tau^{(n)}_i), i=\ov{1,n}\}$  such that
$$
K_n\subset K_{n+1},  \ \cup_nK_n=K, \ \sigma^{(n)}_1<\sigma^{(n)}_2<
\ldots<\sigma^{(n)}_n.
$$
Set
$$
\gamma_n(t)=
\cE_{\tau_n(t)\;t}P\prod_{(\sigma^{(n)}_i, \tau^{(n)}_i)\subset[0, t]}
(P\cE_{\sigma^{(n)}_i \tau^{(n)}_i}P)\cE_{0 \sigma_0},  \ t\geq\sigma_0,
$$
and
$\gamma_n(t)=\cE_{0 t}$ if $t\in[0, \sigma_0).$
Here $\tau_n(t):=\max\{\tau^{(n)}_k: \ \tau^{(n)}_k\leq t
\}.$

  By $\ov{\gamma}(t)$ denote the right-hand side of   \eqref{eq15.2}.
Let us verify that $\ov{\gamma}(t)$ is well-defined and
\begin{equation}
\label{eq17.00}
\lim_{n\to\infty}\gamma_n(t)=\ov{\gamma}(t).
\end{equation}
At first, observe that  \eqref{eq15.3} implies the convergence
$\tau_n(t)\to\tau(t), n\to\infty,$ as $t\geq\sigma_0.$  Thus
$$
\cE_{\tau_n(t)\;t}\to\cE_{\tau(t)\;t}, \  n\to\infty.
$$

Let us prove that
\begin{equation}
\label{eq17.0}
\prod_{(\sigma^{(n)}_i, \tau^{(n)}_i)\subset[0, t]}
(P\cE_{\sigma^{(n)}_i \tau^{(n)}_i
}P)\to
\prod_{(\sigma_i, \tau_i)\subset[0, t]}
(P\cE_{\sigma_i \tau_i}P), \ n\to\infty.
\end{equation}

Observe that for any $d\times d$-matrix   $A$ the matrix $PAP$ can be considered as a linear operator   from $\mbR^{d-1}\times\{0\}$  to $\mbR^{d-1}\times\{0\}.$ In particular, $P$ acts as an identity operator in $\mbR^{d-1}\times\{0\},$ and
\begin{equation}
\label{eq17.1dop}
\|PAP-P\|=\|P(A-E)P\|\leq \|A-E\|.
\end{equation}

For all $s\leq t$ we have an estimate
$$
\|\cE_{st}\|\leq \exp\{c(t-s)\},
$$
where $\|\cdot\|$  is a norm of matrix considered as a linear operator  in
$\mbR^d, \ c=\sup_{r\geq0}\|\alpha(r)\|.$  So, for any $T>0$ there is a constant $K=K(T)$ such that
\begin{equation}
\label{eq17.1}
\forall \ s, t\in[0; T], s\leq t: \
\|\cE_{s t}-E\|=\|\int^t_s\alpha(z)\cE_{s z}dz\|\leq
K(t-s).
\end{equation}

Now Theorem \ref{thm3}, \eqref{eq17.1dop},  and  \eqref{eq17.1}  imply
\eqref{eq17.0}. Hence \eqref{eq17.00} is proved.

Let us prove now that  $\ov{\gamma}(t)$ satisfies  \eqref{eq9.1}. 

Denote  $\alpha_n(t)=\alpha(t)\1_{t\in\cup_k[\sigma_k^{(n)}, \tau_k^{(n)})\cup[0,
\sigma_0)}.$ Observe that
$$
\frac{d\gamma_n(t)}{dt}=\alpha(t)\gamma_n(t), \
t\in\cup^n_{k=1}(\sigma^{(n)}_k, \tau^{(n)}_k)\cup[0, \sigma_0)
$$
and
$$
\gamma_n(\sigma^{(n)}_{k+1}-)=\gamma_n(\tau^{(n)}_k),
\gamma_n(\sigma^{(n)}_{k+1})=P\gamma_n(\tau^{(n)}_k).
$$
So
\begin{equation}
\label{eq19.1}
P\gamma_n(t)=P+P\int^t_0\alpha_n(s)\gamma_n(s)ds, \ t\geq0,
\end{equation}
\begin{equation}
\label{eq19.2}
Q\gamma_n(t)=\begin{cases}
Q+Q\int^t_0\alpha_n(s)\gamma_n(s)ds, \ t<\sigma_0,\\
Q\int^t_{\tau_n(t)}\alpha_n(s)\gamma_n(s)ds, \ t\geq\sigma_0.
\end{cases}
\end{equation}

It follows from \eqref{eq15.3}  that  $\alpha_n(t)\to\alpha(t),
n\to\infty,$  for  $\lambda$-a.a. $t\geq0.$  Thus, the Lebesgue dominated convergence theorem and
\eqref{eq17.00}  yield
$$
P\ov{\gamma}(t)=\lim_{n\to\infty}P\ov{\gamma}_n(t)=
\lim_{n\to\infty}
(P+P\int^t_0\alpha_n(s)\gamma_n(s)ds)=
$$
$$
=P+P\int^t_0\alpha(s)\ov{\gamma}(s)ds,
$$
i.e. $\ov{\gamma}(t)$ satisfies  \eqref{eq9.1}.

Let us show that  $\ov{\gamma}(t)$   satisfies \eqref{eq9.2}. Let $t\geq\sigma_0$ (the case $t\in[0;\sigma_0)$ is trivial).  Since
$\lim_{n\to\infty}\tau_n(t)=\tau(t)$, using the Lebesgue theorem again, we get
$$
Q\gamma_n(t)=Q\int^t_{\tau_n(t)}\alpha_n(s)\gamma_n(s)ds=
$$
$$
=
Q\int^t_0\1_{s\in[\tau_n(t),t]}\alpha_n(s)\gamma_n(s)ds
\underset{n\to\infty}
{\rightarrow}
Q\int^t_0\1_{s\in[\tau(t),t]}\alpha(s)\ov{\gamma}(s)ds
=
$$
$$
=
Q\int^t_{\tau(t)}\alpha(s)\ov{\gamma}(s)ds,
$$
i.e.
$$
Q\ov{\gamma}(t)=
Q\int^t_{\tau(t)}\alpha(s)\ov{\gamma}(s)ds.
$$
Thus, $\ov{\gamma}(t) $ satisfies  \eqref{eq9.1}, \eqref{eq9.2}.
Uniqueness of  \eqref{eq9.1}, \eqref{eq9.2} solution implies the equality $\ov{\gamma}(t)=\gamma(t).$  Theorem \ref{thm4}  is proved.

{\bf Remark.}  Let $\vf_t(x)$  be a solution of reflected SDE 
$$
d\vf_t(x)=a(\vf_t(x))dt+\sum^m_{k=1}\sigma_k(\vf_t(x))dw_k(t)+nL(dt, x), \
t\geq0, x\in\mbR^d,
$$
where conditions \eqref{eq6.2} -- \eqref{eq6.5} are also satisfied. Assume that functions $a, \sigma_k$  are continuously differentiable and have bounded derivatives. Then the Sobolev derivative $\nabla\vf_t(\cdot)$ exists a.s. (see \cite{PilTBiMC}) and there is a   modification
 $\psi_t(x)$ of the derivative such that
\begin{equation}
\label{eq31.111}
\psi_t(x)=\pi\left(
E+
\int^\cdot_0\nabla a(\vf_s(x))\psi_s(x)ds+
\sum^m_{k=1}
\int^\cdot_0\nabla \sigma_k(\vf_s(x))\psi_s(x)dw_k(s)\right)(t),
\end{equation}
where $\pi$ is defined in \eqref{eq11.0}.

The author does not know a result on representation of \eqref{eq31.111} solution as a product \eqref{eq16.0},  where $\cE_{s t}(x)$ is a stochastic exponent,
$$
\cE_{s t}(x)=E+
\int^t_s\nabla a(\vf_z(x))\cE_{s z}(x)dz+
\sum^m_{k=1}\int^t_s\nabla \sigma_k(\vf_z(x))\cE_{s z}(x)dw_k(z).
$$
In this case Theorem \ref{thm3} is inapplicable.   
 It is possible that representation \eqref{eq16.0}  is not satisfied.

\end{proof}

\textbf{\S\, 2. An example of a reflecting flow that is not locally differentiable}

Consider a reflecting flow  $\vf_t(x), t\geq0, x=(x_1, x_2)\in\mbR^2_+$  in  a half-plane that satisfies \eqref{eq6.1}--\eqref{eq6.5} with
$a(x)=Ax,$ where $A=\begin{pmatrix}1&1\\
1&1
\end{pmatrix}:$
\begin{equation}
\label{eq21.1}
d\vf_t(x)=A\vf_t(x)dt+dw(t)+nL(dt, x).
\end{equation}

In coordinate form  equation  \eqref{eq21.1} can be written as follows
\begin{equation}
\label{eq21.2}
d\vf^1_t(x)=(\vf^1_t(x)+\vf^2_t(x))dt+dw_1(t),
\end{equation}
\begin{equation}
\label{eq21.3}
d\vf^2_t(x)=(\vf^1_t(x)+\vf^2_t(x))dt+dw_2(t)+L(dt, x).
\end{equation}
In this case the operator $\cE_{s t}$  from \eqref{eq16.1} is non-random and it is equal  to
$\cE_{s t}=\cE_{t-s},$  where
$$
\cE_t=e^{At}=\begin{pmatrix}
\frac{e^{2t}+1}{2}&\frac{e^{2t}-1}{2}\\
\frac{e^{2t}-1}{2}&\frac{e^{2t}+1}{2}
\end{pmatrix}.
$$

Let $\sigma=\sigma(x), \tau(t)=\tau(t, x),$
$[\sigma_k,\tau_k]=[\sigma_k(x), \tau_k(x)]$
are the same as in \eqref{eq_dop32}.

Denote $\Delta_k=\Delta_k(x)=\tau_k(x)-\sigma_k(x).$  Then for a.a.
$\omega\in \{t\geq\sigma(x)\}$  representation
\eqref{eq16.0} has a form
$$
\nabla\vf_t(x)=
\begin{pmatrix}
\frac{e^{2(t-\tau(t, x))}+1}{2}&0\\
\frac{e^{2(t-\tau(t, x))}-1}{2}&0
\end{pmatrix}
\
\begin{pmatrix}
\prod_{[\sigma_k, \tau_k]\subset(0, t)}\frac{e^{2\Delta_k}+1}{2}&0\\
0&0
\end{pmatrix}
\
\begin{pmatrix}
\frac{e^{2\sigma(x)}+1}{2}&\frac{e^{2\sigma(x)}-1}{2}\\
0&0
\end{pmatrix}.
$$
In particular,
\begin{equation}
\label{eq22.1}
\frac{\pt\vf^1_t(x)}{\pt x_1}=
\frac{e^{2(t-\tau(t))}+1}{2}
\frac{e^{2t\sigma}+1}{2}
\prod_{[\sigma_k, \tau_k]\subset(0, t)}
\frac{e^{2\Delta_k}+1}{2}.
\end{equation}

{\bf Remark.} The  product   in the right-hand side of \eqref{eq22.1} is a product of numbers (and not a product of matrices as in general case of \S\,1). So, generally, the order of the product is inessential.

 By $f_t(x)=f_t(x_1, x_2), \ t\geq\sigma(x),$ denote the right-hand side of  \eqref{eq22.1}.  
   Set $f_t(x)=({e^{2t}+1})/{2}$   for $t<\sigma(x)$.

The main result of this Section is contained in the following two theorems.

\begin{thm}
\label{thm5}
For all  $t>0$

1) for all $x_2>0$  and a.a. $\omega\in\Omega$ the function
$x_1\mapsto f_t(x_1, x_2)$  is nondecreasing in $x_1$.

2) For  any $x=(x_1,x_2), \ x_2>0$ and a.a. $\omega\in\{\sigma(x)<t\}$ a function
$f_t(\cdot)$ is discontinuous in any neighborhood of $x.$
%

Moreover, for any $x=(x_1,x_2), \ x_2>0$ and $\delta>0$:

$P\Bigl(\{ \mbox{a function }   f_t(\cdot,x_2)  \mbox{ does not have a jump  discontinuity on }(x_1-\delta, x_1+\delta)\}\cap\{\sigma(x)<t\}\Bigr)=0$.

\end{thm}

\begin{thm}
\label{thm5_1}
For any $t>0$
$$
\begin{aligned}
&P\Bigl( \mbox{derivative} \ \frac{\pt\vf^1_t(\cdot)}{\pt x_1} \ \mbox{exists for all points  of some }\\
& \mbox{ non-empty open subset of} \ \{x\in\mbR^2_+: \sigma(x)<t\}\Bigr)=0.
\end{aligned}
$$
\end{thm}

%
%

To prove the Theorems we need the following statement on the monotonicity of the flow  $\{\vf_t(x)\}.$

\begin{lem}
\label{lem1}
The flow $\vf_t(\cdot)$ is monotonous in
$x$ in the following sense.

1) If $x^i=(x_1^{(i)}, x_2^{(i)})\in\mbR^2_+, \ i=1,2$
are such that
$x^{(1)}_1\leq x_1^{(2)}, $
$x^{(1)}_2\leq x_2^{(2)}, $ then with probability 1
\begin{equation}
\label{eq23.1}
\vf^1_t(x^1)\leq\vf^1_t(x^2), \
\vf^2_t(x^1)\leq\vf^2_t(x^2)
\end{equation}
for all  $t\geq0.$

2) If we have at least one strict inequality
 $x^{(1)}_1< x_1^{(2)}$ or
$x^{(2)}_2< x_2^{(2)}, $ then
\begin{equation}
\label{eq24.2_1}
\vf^1_t(x^1)<\vf^1_t(x^2), \ t\geq 0,
\end{equation}
$$
\sigma(x^1)<\sigma(x^2),
$$
and
\begin{equation}
\label{eq24.2}
\vf^2_t(x^1)<\vf^2_t(x^2)
\end{equation}
for all $t$ such that $\vf^2_t(x^2)>0.$
\end{lem}

The proof of the first statement can be done similarly to  \cite{PieraMazumdar}.

The proof of \eqref{eq24.2_1} follows from the next obvious lemma.

%
%

\begin{lem}
\label{lem1_2} Assume that continuous functions  $v, \xi_i, g_i,\ i=1,2,$ are such that
$$
\xi_i(t)=\xi_i(0)+\int^t_0g_i(s)ds+v(t), \ t\geq0,
$$
and $\xi_1(0)\leq\xi_2(0),\ g_1(s)\leq g_2(s), s\in[0;t].$ Suppose that either   $\xi_1(0)<\xi_2(0)$ or there exists a point $s_0\in[0;t]$ such that  $ g_1(s_0)< g_2(s_0)$. Then $ \xi_1(t)< \xi_2(t)$.
\end{lem}
The rest of the proof of Lemma \ref{lem1} follows from Lemma \ref{lem1_2} and the following observation. If 
$\vf^2_z(x^2)>0, z\in (t_0;t_1),$ then $\vf^2_z(x^2)$ satisfies the integral equation without reflection on $[t_0,t_1]$:
$$
\vf^2_t(x)=\vf^2_{t_1}(x)+\int_{t_0}^t(\vf^1_z(x)+\vf^2_z(x))dz+(w_2(t)-w_2(t_0)), \ t\in[t_0,t_1].
$$

{\bf Remark.} Generally speaking, it is not difficult to prove deterministic analogue of Lemma \ref{lem1}, where $w$ is an arbitrary continuous function in equation
 \eqref{eq21.1}.

Lemma \ref{lem1} yields the following.

\begin{corl}
\label{corl1}
Let $x_1<y_1.$  Then for any 
$x_2>0, t\geq 0,$ we have inclusion of sets
$$
\{s\in[0;t]: \ \vf^2_s(x_1, x_2)=0\}\supset\{s\in[0;t]: \vf^2_s(y_1, x_2)=0\}.
$$
Moreover, if  $t>\sigma(x_1,x_2)$, then this inclusion is strict.
\end{corl}
\begin{lem} 
\label{lem4}
1) For all  $a_1, a_2>0:$
\begin{equation}
\label{eq26.0}
\frac{e^{a_1}+1}{2}
\cdot
\frac{e^{a_2}+1}{2}
<
\frac{e^{a_1+a_2}+1}{2}.
\end{equation}

2) For all $\{a_n, n\geq1\}\subset(0, \infty),\ \sum_{n\geq1}a_n<\infty:$
\begin{equation}
\label{eq26.1}
\prod_{n\geq1}
\frac{e^{a_n}+1}{2}
<
\frac{\exp\{\sum_{n\geq1}a_n\}+1}{2}.
\end{equation}
\end{lem}

The first statement is trivial. Inequality \eqref{eq26.1}
follows from \eqref{eq26.0} by passing to a limit. It is easy to see that we obtain the strict inequality  in a limit.

 Corollary \ref{corl1} and Lemma \ref{lem4} yield the following  statement.
\begin{corl}
\label{corl2}
Let  $x_1<y_1$  and $t>0.$  Then
$$
f_t(x_1, x_2)\leq f_t(y_1, x_2),
$$
where $f_t$ is the right-hand side of  \eqref{eq22.1}.

If $t>\sigma(x_1,x_2)$, then
\begin{equation}
\label{eq26.3}
f_t(x_1, x_2)< f_t(y_1, x_2).
\end{equation}
\end{corl}

%
Therefore the first part of Theorem \ref{thm5} is proved. Let us verify the second part.

Let $x^0=(x^0_1, x^0_2)\in\mbR\times(0, \infty)$  and $t>0$ be arbitrary. At first let us prove that for a.a. 
$\omega\in\{\sigma(x^0)<t\}$  and for any $\ve>0$ there exists a point $\bar{x}_1\in(x_1^0-\ve,x_1^0)$  and an instant $\bar{t}\in(0;t)$ such that $\vf_.(\bar{x}_1,x_2^0)$ touches the abscissa axis at $\bar{t},$ i.e.
\begin{equation}
\label{eq27touch}
\vf^2_{\bar{t}}(\bar{x}_1,x_2^0)=0,\ \vf^2_{\bar{t}}({x}_1,x_2^0)=0, \ \mbox{for} \ x_1>\bar{x}_1.
\end{equation}
By Girsanov's theorem, the distribution of the process $\vf^2_t(x), t\in[0, T],$ is absolutely continuous with respect to  the distribution of reflected Wiener process that started from  $x_2.$ Hence for a.a.  $\omega$  the set
$\{t\in[0, T]: \vf^2_t(x)=0\}$  is a compact set of zero Lebesgue measure, and it does not have inner and isolated points.
 By Corollary \ref{corl1}, Lemma \ref{lem1}, and the absence of isolated points in the set $\{t\in[0, T]: \vf^2_t(x^0)=0\}$, it follows that for a.a.
$\omega\in\{\sigma(x^0)<t\}$ and any  $x=(x_1, x^0_2)$, $0<x_1<x^0_1,$ there are non-empty intervals
$(\sigma_k(x^0), \tau_k(x^0)),$
$(\sigma_j(x), \tau_j(x)),$  and
$(\sigma_l(x), \tau_l(x))$ (see \eqref{eq_dop32} for the definition of  $(\sigma_j(x), \tau_j(x))$)
  such that
\begin{equation}
\label{eq27.1}
\sigma_k(x_0)<\sigma_j(x)<\tau_j(x)<\sigma_l(x)<\tau_l(x)<\tau_k(x_0)<t.
\end{equation}
We will show the existence of $\bar{x}_1$ and $\bar{t}$ from \eqref{eq27touch} such that  $\bar{x}_1\in(x_1,x_1^0)$
and $\bar{t}\in [\tau_j(x),\sigma_l(x)].$

Put
$$
K_\alpha=\{s\in[0, t]: \vf^2_s(\alpha,x^0_2)=0\}\cap[\tau_j(x), \sigma_l(x)].
$$
Note that

1) $K_{\alpha_1}\subset K_{\alpha_2}, \alpha_1\geq\alpha_2$  (see Corollary \ref{corl1});

2) $K_{x^0_1}=\varnothing;$

3) $K_{x_1}\ne\varnothing.$

Denote $\ov{x}_1=\sup\{\mu: \ K_\mu\ne\varnothing\}.$ 
Since the intersection of centered compact sets is non-empty \cite{Kelli}, we have
$\cap_{\mu<\ov{x}_1}K_\mu\ne\varnothing.$

Suppose $\ov{t}\in \cap_{\mu<\ov{x}_1}K_\mu.$ Then
$\vf^2_{\ov{t}}(\mu,x^0_2)=0, \mu<\ov{x}_1,$ and $\vf^2_{\ov{t}}(\mu,x^0_2)>0, \mu>\ov{x}_1.$  From the continuity  of the flow
$\vf$  in a spatial argument it follows that
$\vf^2_{\ov{t}}(\ov{x}_1,x^0_2)=0.$

Denote $\ov{x}=(\ov{x}_1,x^0_2), \ov{x}^\ve=( \ov{x}_1+\ve,x^0_2).$
  Let
$[\ov{\sigma}^\ve, \ov{\tau}^\ve]$ be a segment from from the collection
$\{[\sigma_m(\ov{x}^\ve), \tau_m(\ov{x}^\ve)]\}$   such that $\bar{t}\in (\ov{\sigma}^\ve, \ov{\tau}^\ve).$ It follows from the choice of $\bar{x},\bar{t},$ that $\vf^2_{{t}}(\ov{x}^\ve)>0$ for all $t\in [\tau_j(x), \sigma_l(x)],$ so $[\tau_j(x), \sigma_l(x)]\subset [\ov{\sigma}^\ve, \ov{\tau}^\ve]. $
It follows from Lemma \ref{lem1} that any segment $[\sigma_i(\ov{x}), \tau_m(\ov{x})]$ is included in some segment from the collection $\{[\sigma_m(\ov{x}^\ve), \tau_m(\ov{x}^\ve)]\}$. Observe that any multiplier in the definition of the function $f_t$ is greater than 1. Therefore, by Lemma \ref{lem4} 
$$
f_t(\ov{x}^\ve)-f_t(\ov{x})\geq
$$
$$
\geq\frac{
e^{2(\ov{\tau}^\ve-\ov{\sigma}^\ve)}+1}{2}-
\frac{e^{2(\ov{\tau}^\ve-\ov{t})}+1}{2}\cdot
\frac{e^{2(\ov{t}-\ov{\sigma}^\ve)}+1}{2}=
$$
$$
=
\frac{(e^{2(\ov{\tau}^\ve-\ov{t})}-1)(e^{2(\ov{t}-\ov{\sigma}^\ve)}-1)}{4}
\geq
\frac{(e^{2(\tau_l(x)-\ov{t})}-1)(e^{2(\ov{t}-\sigma_j(x))}-1)}{4}>0.
$$
Thus
$$
f(\ov{x}_1+,x^0_2)-f(\ov{x}_1,x^0_2)>0.
$$
Theorem \ref{thm5} is proved.

\emph{Proof of Theorem \ref{thm5_1}.}  Let us remember that  for any $t$ and  $x$ we have the equality
\begin{equation}
\label{eq31.1}
f_t(x)= \frac{\pt\vf^1_t(x)}{\pt x_1}
 \end{equation}
for a.a.  $\omega\in \{\sigma(x)\leq t\}.$ A derivative cannot has a jump discontinuity, but a function $f_t(\cdot, x_2)$ has jump discontinuities in any neighborhood of $x$ and a.a.  $\omega\in \{\sigma(x)\leq t\}$ because of Theorem \ref{thm5}. So, if equality \eqref{eq31.1} be satisfied simultaneously for all  $x$ (independently of $\omega$), then this contradiction  would immediately imply the proof of the Theorem. Generally, a set of appropriate $\omega$ depends on $x$. So, to be accurate, we need several additional arguments.

   Fubini's theorem yields that for any  $x^0=(x^0_1, x^0_2), \ve>0$, and for a.a. $\omega\in\{\sigma(x^0)<t\}$ there are  $x_2, |x_2-x^2_0|<\ve,$ and $\delta>0$ such that  $\sigma(x)<t$, $
f_t(x_1,x_2)= \frac{\pt\vf^1_t(x_1,x_2)}{\pt x_1}$ for a.a.  $x_1\in[x_1^0-\delta;x_1^0+\delta]$ with respect to the Lebesgue measure, and $f_t(\cdot,x_2)$ is discontinuous on $(x_1^0-\delta, x_1^0+\delta)$.   The function  $f_t(\cdot, x_2)$ is monotonous on     $[x_1^0-\delta;x_1^0+\delta]$. Hence $\vf^1_t(\cdot, x_2)$ is concave on  $[x_1^0-\delta;x_1^0+\delta]$. So, left and right derivatives of  $\vf^1_t(\cdot, x_2)$ exist and are non-decreasing. Assume that the derivative in  $x_1$ exists for all $x_1\in[x_1^0-\delta;x_1^0+\delta]$. Then it must be monotonous and discontinuous, because the function $f_t(\cdot, x_2)$ is monotonous and discontinuous. This contradiction proves Theorem \ref{thm5_1}.

{\bf Remark.} It follows from the proof of Theorem \ref{thm5} that the absence of the derivative  $\nabla\vf_t(\cdot)$ on a dense subset of
$\{x: \sigma(x)<t\}$ is rather a rule than an exception.
%
%

\end{document}